\documentclass[12pt]{amsart}
\usepackage{amsfonts,amssymb,amscd,amsmath,enumerate,verbatim, fancyvrb}
\usepackage[latin1]{inputenc}
\usepackage{amscd}
\usepackage{latexsym}
\usepackage{mathptmx} 
\usepackage{multicol}
\usepackage{epsfig}
\usepackage{color}
\definecolor{cream}{cmyk}{0,0,0.3,0}
\definecolor{orange}{cmyk}{0,0.61,0.87,0}
\definecolor{whiteblue}{cmyk}{0.2,0,0,0}
\definecolor{whitepeach}{cmyk}{0,0.2,0,0}
\definecolor{peach}{cmyk}{0,0.2,0.2,0}
\definecolor{bluegreen}{cmyk}{0.2,0,0.1,0}
\definecolor{whitepurple}{cmyk}{0.3,0.2,0.1,0}

\input xy
\xyoption{all}

\def\frk{\mathfrak}               

\def\Phi{{\frk N}}
\def\opn#1#2{\def#1{\operatorname{#2}}} 
\opn\chara{char} 
\opn\length{\ell} 
\opn\pd{pd} 
\opn\rk{rk}
\opn\projdim{proj\,dim} 
\opn\injdim{inj\,dim} 
\opn\rank{rank}
\opn\depth{depth} 
\opn\grade{grade} 
\opn\height{height}
\opn\embdim{emb\,dim} 
\opn\codim{codim}

\opn\Tr{Tr} 
\opn\bigrank{big\,rank}
\opn\superheight{superheight}
\opn\lcm{lcm}
\opn\trdeg{tr\,deg}
\opn\reg{reg} 
\opn\lreg{lreg} 
\opn\ini{in} 
\opn\lpd{lpd}
\opn\size{size}
\opn\mult{mult}
\opn\dist{dist}
\opn\cone{cone}
\opn\lex{lex}
\opn\rev{rev}
\opn\div{div} \opn\Div{Div} \opn\cl{cl} \opn\Cl{Cl}
\opn\Spec{Spec} \opn\Supp{Supp} \opn\supp{supp} \opn\Sing{Sing}
\opn\Ass{Ass} \opn\Min{Min}
\opn\Ann{Ann} \opn\Rad{Rad} \opn\Soc{Soc}
\opn\Syz{Syz} \opn\Im{Im} \opn\Ker{Ker} \opn\Coker{Coker}
\opn\Am{Am} \opn\Hom{Hom} \opn\Tor{Tor} \opn\Ext{Ext}
\opn\End{End} \opn\Aut{Aut} \opn\id{id} \opn\ini{in}

\opn\nat{nat}
\opn\pff{pf}
\opn\Pf{Pf} \opn\GL{GL} \opn\SL{SL} \opn\mod{mod} \opn\ord{ord}
\opn\Gin{Gin}
\opn\Hilb{Hilb}\opn\adeg{adeg}\opn\std{std}\opn\ip{infpt}
\opn\Pol{Pol}
\opn\sat{sat}
\opn\Var{Var}
\opn\Gen{Gen}

\opn\aff{aff} \opn\con{conv} \opn\relint{relint} \opn\st{st}
\opn\lk{lk} \opn\cn{cn} \opn\core{core} \opn\vol{vol}
\opn\link{link} \opn\star{star}\opn\indmatch{ind-match}
\opn\gr{gr}

\def\pot#1#2{#1[\kern-0.28ex[#2]\kern-0.28ex]}

\opn\dirlim{\underrightarrow{\lim}}
\opn\inivlim{\underleftarrow{\lim}}
\let\to=\rightarrow

\def\Implies{\ifmmode\Longrightarrow \else
        \unskip${}\Longrightarrow{}$\ignorespaces\fi}
\def\implies{\ifmmode\Rightarrow \else
        \unskip${}\Rightarrow{}$\ignorespaces\fi}
\def\iff{\ifmmode\Longleftrightarrow \else
        \unskip${}\Longleftrightarrow{}$\ignorespaces\fi}

\let\:=\colon
\newtheorem{Theorem}{Theorem}
\newtheorem{Lemma}[Theorem]{Lemma}

\newtheorem{Proposition}[Theorem]{Proposition}

\newtheorem{Example}[Theorem]{Example}

\newtheorem{Question}[Theorem]{Question}

\let\epsilon\varepsilon
\let\phi=\varphi
\let\kappa=\varkappa
\def\qed{\ifhmode\textqed\fi
      \ifmmode\ifinner\quad\qedsymbol\else\dispqed\fi\fi}
\def\textqed{\unskip\nobreak\penalty50
       \hskip2em\hbox{}\nobreak\hfil\qedsymbol
       \parfillskip=0pt \finalhyphendemerits=0}
\def\dispqed{\rlap{\qquad\qedsymbol}}

\opn\dis{dis}
\opn\height{height}
\opn\dist{dist}
\def\pnt{{\raise0.5mm\hbox{\large\bf.}}}

\opn\Lex{Lex}

\begin{document}
\title{Extremal Betti numbers of edge ideals}
\author{Takayuki Hibi, Kyouko Kimura and Kazunori Matsuda}
\address{Takayuki Hibi,
Department of Pure and Applied Mathematics,
Graduate School of Information Science and Technology,
Osaka University, Suita, Osaka 565-0871, Japan}
\email{hibi@math.sci.osaka-u.ac.jp}

\address{Kyouko Kimura,
Department of Mathematics, 
Faculty of Science, 
Shizuoka University, 836 Ohya, Suruga-ku, Shizuoka 422-8529, Japan}
\email{kimura.kyoko.a@shizuoka.ac.jp}

\address{Kazunori Matsuda,
Kitami Institute of Technology, 
Kitami, Hokkaido 090-8507, Japan}
\email{kaz-matsuda@mail.kitami-it.ac.jp}

\begin{abstract}
Given integers $r$ and $b$ with  $1 \leq b \leq r$, a finite simple connected graph $G$ for which $\reg(S/I(G)) = r$ and the number of extremal Betti numbers of $S/I(G)$ is equal to $b$ will be constructed. 
\end{abstract}
\maketitle

Let $S = K[x_1, \ldots, x_n]$ denote the polynomial ring in 
$n$ variables over a field $K$ with each $\deg x_i = 1$ and 
$I \subset S$ a homogeneous ideal.  
Let
\[
{\bf F}_{S/I} : 0 \to \bigoplus_{j \geq 1} S\left(-(p+j)\right)^{\beta_{p, p + j}(S/I)} \to \cdots \to \bigoplus_{j \geq 1} S\left(-(1+j)\right)^{\beta_{1, 1 + j}(S/I)} \to S \to S/I \to 0
\]
be the minimal graded free resolution of $S/I$ over $S$, where $p = \projdim (S/I)$ is the projective dimension of $S/I$ 
and $\beta_{i, i + j}(S/I)$ is the $(i, i + j)$-th graded Betti number of $S/I$. 
The ({\em Castelnuovo--Mumford}\,) {\em regularity} of $S/I$ is  
\[
\reg\left(S/I\right) = \max\{ \, j \, : \, \beta_{i, i + j} (S/I) \neq 0 \, \}.
\] 
A graded Betti number $\beta_{i, i + j}(S/I) \neq 0$ is said to be {\em extremal} 
(\cite[Definition 4.3.13]{HH}) if $\beta_{k, k + \ell}(S/I) = 0$ 
for all pairs $(k, \ell) \neq (i, j)$ with $k \geq i$ and $\ell \geq j$. 
Extremal Betti numbers of graded algebras 
have been studied, for example, in \cite{BCP, MoMo}. 
In general, $S/I$ has the unique extremal Betti number if and only if 
$\beta_{p, p + r}(S/I) \neq 0$, where $r = \reg (S/I)$. 
This equivalent properties are hold 
if $S/I$ is Cohen-Macaulay (\cite[Lemma 3]{BiHe}). 
Moreover a Cohen-Macaulay graded algebra $S/I$ is Gorenstein 
if and only if $\sum_{j} \beta_{p, p + j}(S/I) =  \beta_{p, p + r}(S/I) = 1$. 

\par
Let $G$ be a finite simple graph (i.e. a graph with no loop and no multiple edge) on the vertex set 
$V(G) = \{x_{1}, x_{2}, \ldots, x_{n}\}$ with $E(G)$ its edge set.
The {\em edge ideal} of $G$ is
\[
I(G) = \left( \, x_{i}x_{j}\, : \,  \{x_{i}, x_{j}\} \in E(G) \, \right) \subset S = K[V(G)] = K[x_{1}, x_{2}, \ldots, x_{n}]. 
\]

\par
In the present paper, given integers $r$ and $b$ with $1 \leq b \leq r$, the existence of a finite simple connected graph $G$ for which $\reg(S/I(G)) = r$ and the number of extremal Betti numbers of $S/I(G)$ is equal to $b$ will be shown. 

\begin{Theorem}
\label{ExtremalBettiRegularity}
  Let $r,b$ be integers with $1 \leq b \leq r$. 
  Then there exists a finite simple connected graph $G_{r,b}$ with 
  $\reg K[V(G_{r,b})]/I (G_{r,b}) = r$ such that 
  the number of extremal Betti numbers of $K[V(G_{r,b})]/I(G_{r,b})$ 
  is equal to $b$. 
\end{Theorem}



In order to prove Theorem \ref{ExtremalBettiRegularity}, 
we use the non-vanishing theorem 
for Betti numbers \cite[Theorem 4.1]{Kimura12}. 
As before, we recall fundamental materials on graph theory 
which is needed to understand the theorem. 

\par
Let $G$ be a finite simple graph on the vertex set $V$ and $W$ a subset of $V$.  The induced subgraph of $G$ on $W$ is the subgraph $G_W$ of $G$ with the vertex set $V(G_W) = W$ and with the edge set $E(G_W) = \big\{ \{x_i, x_j\} \in E(G) \, : \, x_i, x_j \in W \, \big\}$.  A finite simple graph $G$ is {\em chordal} if every cycle in $G$ of length $> 3$ has a chord. 
A subset $M$ of $E(G)$ is a {\em matching} of $G$ if, for any $e, e' \in M$ with $e \neq e'$, one has $e \cap e' = \emptyset$. 
A matching $M$ is called an {\em induced matching} of $G$ if, for any $e, e' \in M$ with $e \neq e'$, there is no $e'' \in E(G)$ satisfying both of $e \cap e'' \neq\emptyset$ and $e' \cap e'' \neq\emptyset$. 
The {\em induced matching number} $\indmatch(G)$ of $G$ is the maximum size of induced matching of $G$. 
A complete bipartite graph of type $(1,d)$ is called a {\em bouquet}. 
Let $\mathcal{B} = \{ B_1, \ldots, B_s \}$ be a set of bouquets, where 
$B_i$ is a subgraph of $G$, and set $V(\mathcal{B}) = V(B_1) \cup \cdots \cup V(B_s)$. 
We say that $\mathcal{B}$ is a {\em strongly disjoint set of bouquets} of $G$
if $V(B_k) \cap V(B_{\ell}) = \emptyset$ for all $k \neq \ell$ 
and if, for each $1 \leq k \leq s$, there exists $e_k \in E(B_k)$ such that $\{ e_1, \ldots, e_s \}$ forms an induced matching of $G$. 
When $\mathcal{B}$ is a strongly disjoint set of bouquets, we define the type of $\mathcal{B}$ as 
$(|V(\mathcal{B})| - s, s)$. 
Finally, we say that $G$ contains a strongly disjoint set of bouquets of type $(i,j)$ if there exists a strongly disjoint set of bouquets $\mathcal{B}$ of $G$ whose type is $(i,j)$ and which satisfies $V(\mathcal{B}) = V(G)$.

\begin{Proposition}[{\cite[Theorem 4.1]{Kimura12}}]
  \label{NonVanishingBetti}
  Let $G$ be a finite simple graph on $V$ and $S=K[V]$. 
  Suppose that $G$ is chordal. 
  Then $\beta_{i, i+j} (S/I(G)) \neq 0$ if and only if there exists a subset 
  $W$ of $V$ such that the induced subgraph $G_W$ contains a strongly disjoint 
  set of bouquets of type $(i,j)$. 
\end{Proposition}

\smallskip

We now turn to a proof of Theorem \ref{ExtremalBettiRegularity}.

\smallskip

\noindent 
{\bf (First Step)} Let $1 = b \leq r$ and $G_{r, 1}$ the graph consisting 
of $r$ paths $P_{3}$ with the common vertex $z$, that is, 
\begin{eqnarray*}
V(G_{r, 1}) &:=& \{ x_{1}, \ldots, x_{r} \} \cup \{ y_{1}, \ldots, y_{r} \} \cup \{ z \}, \\
E(G_{r, 1}) &:=& \big\{ \{ z, y_{i} \} : i = 1, \ldots, r \big\} \cup \big\{ \{ x_{i}, y_{i} \} : i = 1, \ldots, r \big\}. 
\end{eqnarray*}
Note that the graph $G_{r, 1}$ is a tree, in particular a chordal graph. 
Hence we can apply Proposition \ref{NonVanishingBetti}. 

\par
For each $i = 1, \ldots, r - 1$, let $B_i$ be the induced subgraph 
of $G_{r,1}$ on $\{x_{i}, y_{i}\}$. 
Also let $B_r$ be the induced subgraph of $G_{r,1}$ on $\{x_{r}, y_{r}, z\}$. 
Then $\mathcal{B} = \{ B_1, \ldots, B_r \}$ is a strongly disjoint set 
of bouquets of $G_{r, 1}$ of type $(r+1, r)$ 
(notice that $\{x_{i}, y_{i}\} \in E(B_{i})$ $(i = 1, \ldots, r)$ 
form an induced matching of $G_{r, 1}$). 
Therefore, by virtue of Proposition \ref{NonVanishingBetti}, 
one has 
\begin{displaymath}
  \beta_{r+1, (r+1) + r}(K[V(G_{r,1})]/I(G_{r, 1})) \neq 0 
\end{displaymath}
and this is the only extremal Betti number of $K[V(G_{r,1})]/I(G_{r, 1})$. 
It then also follows that $\reg(K[V(G_{r,1})]/I(G_{r, 1})) = r$.


%

\smallskip

\noindent 
{\bf (Second Step)} 
Let $r$ and $b$ be integers with $2 \leq b \leq r$. 
We then introduce the finite simple connected graphs $G_{r,j}$ for $j= 2, \ldots, b$ constructed as follows. 
Starting with the star triangle consisting of 
$r$ triangles with the common vertex $z$, 
we define $G_{r,2}$ by adding a new vertex $w_1$ joining 
with all vertices of a triangle, say, $z, x_1, y_1$, to it. 
Thus 
\begin{displaymath}
  \begin{aligned}
    V(G_{r,2}) &:= 
      \{ x_1, \ldots, x_r \} \cup \{ y_1, \ldots, y_r \} 
      \cup \{ z \} \cup \{ w_1 \}, \\
    E(G_{r,2}) &:= 
      \big\{ \{ z, x_i \} \; : \; i = 1, \ldots, r \big\} 
       \cup \big\{ \{ z, y_i \} \; : \; i = 1, \ldots, r \big\} 
       \cup \big\{ \{ x_i, y_i \} \; : \; i = 1, \ldots, r \big\} \\
      &\qquad 
      \cup \big\{ \{ w_1, x_1 \}, \; \{ w_1, y_1 \}, \; \{ w_1, z \} \big\}. 
  \end{aligned}
\end{displaymath}
Inductively, we define the finite simple connected graph $G_{r,j+1}$ for each $j = 2, \ldots, b-1$ 
by adding a new vertex $w_{j}$ joining with vertices 
\begin{displaymath}
  z, \; x_1, \ldots, x_{j}, \; y_1, \ldots, y_{j}, \; w_1, \ldots, w_{j-1} 
\end{displaymath}
to $G_{r, j}$. In other words,
\begin{displaymath}
  \begin{aligned}
    V(G_{r,j+1}) &:= 
      \{ x_1, \ldots, x_r \} \cup \{ y_1, \ldots, y_r \} 
      \cup \{ z \} \cup \{ w_1, \ldots, w_{j} \}, \\
    E(G_{r,j+1}) &:= E(G_{r,j}) \cup \big\{ \{ w_{j}, z \} \big\} 
      \cup \big\{ \{ w_{j}, w_i \} \; : \; i = 1, \ldots, j-1 \big\} \\
     &\qquad \cup \big\{ \{ w_{j}, x_i \} \; : \; i = 1, \ldots, j \big\} 
      \cup \big\{ \{ w_{j}, y_i \} \; : \; i = 1, \ldots, j \big\}. 
  \end{aligned}
\end{displaymath}


Then $G_{r,b}$ ($2 \leq b \leq r$) has the following properties. 
\begin{Lemma}
  \label{Gbr}
  Let $2 \leq b \leq r$ and $S = K[V(G_{r,b})]$. Then 
  \begin{enumerate}
  \item[{\rm (1)}] the graph $G_{r,b}$ is chordal\,{\rm ;} 
  \item[{\rm (2)}] $\reg (S/I(G_{r,b})) = r$\,{\rm ;}
  \item[{\rm (3)}] $\projdim (S/I(G_{r,b})) = 2r+b-1$\,{\rm ;}
  \end{enumerate}
\end{Lemma}
\begin{proof}
%
%
  \begin{enumerate}
  \item Let $j$ be an integer with $1 \leq j \leq b-1$. 
    We prove that $G_{r,j+1}$ is chordal under the hypothesis that 
    $G_{r,j}$ is chordal. 
    Here, for abuse notation, we use $G_{r, 1}$ to denote 
    the star triangle consisting of $r$ triangles, 
    obviously it is a chordal graph. 
    Suppose that there exists a cycle $C$ of $G_{r,j+1}$ 
    of length $> 3$ with no chord. 
    Since $G_{r,j}$ is chordal, it follows that $C$ must contain $w_j$. 
    If $C$ contains a vertex which does not belong to $N_{G_{r, j+1}} [w_j]$, 
    where $N_{G_{r, j+1}} [w_j] = N_{G_{r, j+1}} (w_j) \cup \{w_{j}\}$ 
    and $N_{G_{r, j+1}} (w_j)$ is the neighbourhood of $w_{j}$ in $G_{r, j+1}$, 
    then $C$ must be a triangle $z x_k y_k$ for some $k \geq j+1$. 
    This is a contradiction. 
    Thus all vertices of $C$ must belong to $N_{G_{r,j+1}} [w_j]$. 
    Since $C$ contains $w_j$, this contradicts the hypothesis that 
    $C$ does not have a chord. 
  \item Since $G_{r,b}$ is chordal, the regularity of $S/I(G_{r,b})$ 
    is equal to the induced 
    matching number $\indmatch(G_{r,b})$ of $G_{r,b}$ (\cite[Corollary 6.9]{HaVanTuyl}). 
    Since $\{ x_1, y_1 \}, \ldots, \{ x_r, y_r \}$ form an induced matching, 
    one has $\indmatch(G_{r,b}) \geq r$. 
    Conversely, let $\mathcal{M}$ be an induced matching of $G_{r,b}$ which is 
    different from the above one. 
    If one of edges in $\mathcal{M}$ contains $z$, 
    then $|\mathcal{M}| = 1$ because 
    $N_{G_{r,b}} [z] = V(G_{r,b})$. 
    Otherwise, one of edges in $\mathcal{M}$ contains $w_j$ for some $j$. 
    Then the number of edges of $\mathcal{M}$ contained in $G_{N[w_j]}$ 
    is $1$. The other edges of $\mathcal{M}$ must be contained in 
    $\{ \{ x_{j+1}, y_{j+1} \}, \ldots, \{ x_r, y_r \} \}$. 
    It then follows that $|\mathcal{M}| \leq r-j+1 \leq r$. 
    Hence $\indmatch(G_{r,b}) = r$, as required. 
  \item Since $G_{r,b}$ is chordal, 
    we can apply Proposition \ref{NonVanishingBetti}. 
    One has $|V(G_{r,b})| = 2r+b$ and the bipartition 
    $\{ z \} \sqcup (V(G_{r,b}) \setminus \{ z \})$ of the vertex set 
    defines a bouquet of type $(2r+b-1,1)$. 
    Hence $\beta_{2r+b-1, (2r+b-1)+1} (S/I(G_{r,b})) \neq 0$ and  
    $\projdim (S/I(G_{r,b})) = 2r+b-1$, as desired. \qed
  \end{enumerate}
\end{proof}

Finally we investigate the extremal Betti numbers of 
$K[V(G_{b,r})]/I(G_{r,b})$. 
The following lemma completes (Second Step) of our proof. 
\begin{Lemma}
  \label{GbrExtremal}
  Let $2 \leq b \leq r$ and $S = K[V(G_{r,b})]$. 
  Then the extremal Betti numbers of $S/I(G_{r,b})$ are
  \[
  \beta_{r+b+i-1, (r+b+i-1)+(r-i+1) } (S/I(G_{r,b})), \, \, \, \, \, \, \, \, \, \, i = 1, 2, \ldots, b - 1, 
  \]
and \[ \beta_{2r+b-1, (2r+b-1)+1 } (S/I(G_{r,b})). \] 
\end{Lemma}
\begin{proof}
  We first note that, in general, 
  $\beta_{k, k+\ell} (S/I(G)) = 0$ when $k + \ell > |V(G)|$. 
  In our case, since
  $|V(G_{r,b})| = 2r+b = (r+b+i-1) + (r-i+1)$, 
  if 
  \[
    \beta_{r+b+i-1, (r+b+i-1) + (r-i+1)} (S/I(G_{r,b})) \neq 0,
  \]
  then it follows that this Betti number is 
  extremal. We have already proven that 
  \[
    \beta_{2r+b-1, (2r+b-1)+1} (S/I(G_{r,b})) \neq 0
  \]
  in the proof of Lemma \ref{Gbr} (3). 
  We prove 
  \[
    \beta_{r+b+i-1, 2r+b} (S/I(G_{r,b})) 
    = \beta_{r+b+i-1, (r+b+i-1) + (r-i+1)} (S/I(G_{r,b})) \neq 0
  \] 
  for $i=1, 2, \ldots, b-1$. 

  \par
  For $i = 1, 2, \ldots, b-1$, consider the following set 
  $\mathcal{B}_{i} = \{ B_1^{(i)}, \ldots, B_{r-i+1}^{(i)} \}$ of bouquets 
  (we describe a bouquet by its vertex bipartition): 
  \begin{displaymath}
    \begin{aligned}
      B_1^{(i)} &: \{ w_i \} \sqcup N_{G_{r,b}} (w_i) \\
                &\quad = \{ w_i \} \sqcup 
          (\{ z, \; x_1, \ldots, x_{i}, \; y_1, \ldots, y_i, \; 
              w_1, \ldots, w_{b-1} \} \setminus \{ w_i \}), \\
      B_{k+1}^{(i)} &: \{ x_{i+k} \} \sqcup \{ y_{i+k} \}, 
          \quad k = 1, 2, \ldots, r-i. 
    \end{aligned}
  \end{displaymath}
  Since 
  \begin{displaymath}
    \{ w_i, x_i \} \in E(B_1^{(i)}), \; 
    \{ x_{i+1}, y_{i+1} \} \in E(B_{2}^{(i)}), \ldots, 
    \{ x_r, y_r \} \in E(B_{r-i+1}^{(i)}) 
  \end{displaymath}
  form an induced matching of $G_{r,b}$, 
  it follows that $\mathcal{B}_i$ is a strongly disjoint set 
  of bouquets contained in $G_{r,b}$ of type $(r+b+i-1, r-i+1)$. 
  Proposition \ref{NonVanishingBetti} says
  $\beta_{r+b+i-1, 2r+b} (S/I(G_{r,b})) \neq 0$. 

  \par
  To complete the proof, 
  our work is to show that 
  \[
    \beta_{r+2b-2 + i, (r+2b-2+i)+j} (S/I(G_{r,b})) = 0
  \] 
  for $i = 1, 2, \ldots, r-b$ and  $j = 2, 3, \ldots, r-b-i+2$, 
  since we know $\beta_{r+2b-2+i, (r+2b-2+i)+j} (S/I(G_{r,b})) = 0$ 
  when $(r+2b-2+i)+j > 2r+b$, i.e., $i + j > r - b + 2$. 

  \par
  We proceed the proof by induction on $r-b$ $(\geq 0)$. 
  When $r-b=0$, there is nothing to prove. 

  \par
  Assume $r-b > 0$. 
  We show that there is no set of bouquets 
  which guarantees the non-vanishing of these Betti numbers 
  in meaning of Proposition \ref{NonVanishingBetti}. 
  On the contrary, suppose that there exists a strongly disjoint set 
  of bouquets $\mathcal{B}$ of $G_{r,b}$ contained in 
  $(G_{r,b})_{V(\mathcal{B})}$ 
  of type $(r+2b-2+i, j)$, 
  where $1 \leq i \leq r-b$ and $2 \leq j \leq r-b-i+2$. 
  Let $\mathcal{B} = \{ B_1, \ldots, B_j \}$ and 
  assume that $e_1, \ldots, e_j$ form an induced matching of $G_{r,b}$ 
  with each $e_{\ell} \in E(B_{\ell})$. 

  \par
  (Step 1) Assume that $\{ x_k, y_k \} \notin \{ e_1, \ldots, e_j \}$
  for some $k$ with $b \leq k \leq r$. 
  In this case, $x_k, y_k \notin V(\mathcal{B})$ because $j \geq 2$. 
  Then $\mathcal{B}$ can be regarded as a strongly disjoint set of bouquets 
  of $G_{r-1, b}$ of 
  \[
    {\rm type} \, (r+2b-2+i, j) = {\rm type} \, ((r-1)+2b-2 + (i+1), j). 
  \]
  Since 
  \begin{displaymath}
    \begin{aligned}
      &\beta_{r-1 + 2b - 2 + \alpha, (r-1 + 2b-2 + \alpha) + \beta} 
         (K[V(G_{r-1,b})]/I(G_{r-1,b})) = 0, \\
      &\qquad \alpha = 1, 2, \ldots, r-1-b; 
         \quad \beta = 2, 3, \ldots, (r-1) - b - \alpha + 2 
    \end{aligned}
  \end{displaymath}
  by inductive hypothesis, the possible pairs $(i,j)$ can be 
  \begin{itemize}
  \item$(i, r-b-i+2)$, $(i, r-b-i+1)$ for $1 \leq i \leq r-b-2$; 
  \item$i=r-b-1$, then $j= 2,3$; 
  \item$i=r-b$, then $j=2$. 
  \end{itemize}
  In each of the three cases, $i+j$ is equal to either $r-b+1$ or $r-b+2$. 
  Hence
  $|V(\mathcal{B})| = (r+2b-2+i)+j$ is equal to either $2r+b-1$ or $2r+b$, 
  which contradicts $x_k, y_k \notin V(\mathcal{B})$. 

  \par
  (Step 2) Assume that $\{ x_k, y_k \} \in \{ e_1, \ldots, e_j \}$ for all  
  $k = b, \ldots, r$. 
  Then $j \geq r-b+1$. Since $j \leq r-b-i+2$ and $i \geq 1$, 
  it follows that $j=r-b+1$ and $i=1$. 
  Hence 
  \begin{displaymath}
    V(\mathcal{B}) \subset 
      \{ z \} \cup \{ x_b, \ldots, x_r \} \cup \{ y_b \ldots, y_r \}. 
  \end{displaymath}
  This contradicts the fact that the type of $\mathcal{B}$ is  
  $(r+2b-2+1, r-b+1)$. 
\end{proof}


\begin{Example}
Let $b = 3$ and $r = 5$. 
The graph $G_{5, 3}$ is the following\,$:$

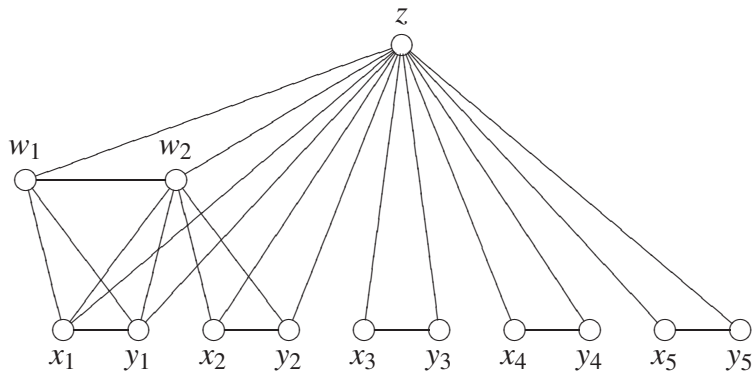
\bigskip

\begin{xy}
	\ar@{} (0,0);(30, -8) *++!U{x_{1}} *\cir<4pt>{} = "X1";
	\ar@{-} "X1";(40, -8) *++!U{y_{1}} *\cir<4pt>{} = "Y1";
	\ar@{} (0,0);(50, -8) *++!U{x_{2}} *\cir<4pt>{} = "X2";
	\ar@{-} "X2";(60, -8) *++!U{y_{2}} *\cir<4pt>{} = "Y2";
	\ar@{} (0,0);(70, -8) *++!U{x_{3}} *\cir<4pt>{} = "X3";
	\ar@{-} "X3";(80, -8) *++!U{y_{3}} *\cir<4pt>{} = "Y3";
	\ar@{} (0,0);(90, -8) *++!U{x_{4}} *\cir<4pt>{} = "X4";
	\ar@{-} "X4";(100, -8) *++!U{y_{4}} *\cir<4pt>{} = "Y4";
	\ar@{} (0,0);(110, -8) *++!U{x_{5}} *\cir<4pt>{} = "X5";
	\ar@{-} "X5";(120, -8) *++!U{y_{5}} *\cir<4pt>{} = "Y5";
	\ar@{} (0,0);(75, 30) *++!D{z} *\cir<4pt>{} = "Z";
	\ar@{-} "Z";"X1";
	\ar@{-} "Z";"X2";
	\ar@{-} "Z";"X3";
	\ar@{-} "Z";"X4";
	\ar@{-} "Z";"X5";
	\ar@{-} "Z";"Y1";
	\ar@{-} "Z";"Y2";
	\ar@{-} "Z";"Y3";
	\ar@{-} "Z";"Y4";
	\ar@{-} "Z";"Y5";
	\ar@{} (0,0);(45, 12) *++!D{w_{2}} *\cir<4pt>{} = "W2";
	\ar@{-} "W2";"X1";
	\ar@{-} "W2";"X2";
	\ar@{-} "W2";"Y1";
	\ar@{-} "W2";"Y2";
	\ar@{-} "W2";"Z";
	\ar@{} (0,0);(25, 12) *++!D{w_{1}} *\cir<4pt>{} = "W1";
	\ar@{-} "W1";"X1";
	\ar@{-} "W1";"Y1";
	\ar@{-} "W1";"W2";
	\ar@{-} "W1";"Z";
\end{xy}

\bigskip

The Betti table of $K[V(G_{5,3})] / I(G_{5, 3})$ is 

\bigskip

\centering
{\rm 
\begin{BVerbatim}
1  .  .   .   .   .   .   .   .   .  .  . .
. 24 94 248 512 798 925 792 495 220 66 12 1 
.  . 33  86  91  53  18   3   .   .  .  . .
.  .  .  37 100 105  57  18   3   .  .  . . 
.  .  .   .  18  49  49  23   6   1  .  . . 
.  .  .   .   .   3   8   7   2   .  .  . . 
\end{BVerbatim}
}
\end{Example}


\begin{Question}
\label{Q}
It is natural to ask if, given integers $1 \leq b \leq r$ and $b \leq p$, there exists a finite simple connected graph $G_{p,r,b}$ for which
\begin{itemize}
\item
$\reg(K[V(G_{p,r,b})]/I(G_{p,r,b})) = r$\,{\rm ;}
\item
$\projdim(K[V(G_{p,r,b})]/I(G_{p,r,b})) = p$\,{\rm ;}
\item
the number of extremal Betti numbers of $K[V(G_{p,r,b})]/I(G_{p,r,b})$ is equal to $b$.
\end{itemize}
\end{Question}  

\begin{Example}
{\em
Let $1 = b \leq r < p$ and $G_{p, r, 1}$ be the connected graph for which 
\begin{eqnarray*}
V(G_{p, r, 1}) &=& \{ x_{1}, \ldots, x_{p-1} \} \cup \{ y_{1}, \ldots, y_{r} \} \cup \{z\}, \\
E(G_{p, r, 1}) &=& \big\{ \{ z, y_{i} \} : i = 1, \ldots, r \big\} \cup \big\{ \{ x_{i}, y_{i} \} : i = 1, \ldots, r - 1 \big\} \\
& & \qquad  \cup \big\{ \{ x_{j}, y_{r} \} : j = r, \ldots, p - 1 \big\}. 
\end{eqnarray*}
Note that $G_{r+1, r, 1} = G_{r, 1}$, appears in First Step of the proof of Theorem \ref{ExtremalBettiRegularity}. 
Then Proposition \ref{NonVanishingBetti} says that $\beta_{p, p + r}(K[V(G_{p,r,1})]/I(G_{p, r, 1}))$ is the only 
extremal Betti number of $K[V(G_{p,r,1})]/I(G_{p, r, 1})$. 
In particular, 
\begin{itemize}
	\item $\reg (K[V(G_{p,r,1})]/I(G_{r+1, r, 1})) = r$\,{\rm ;}
	\item $\projdim (K[V(G_{p,r,1})]/I(G_{r+1, r, 1})) = p$\,{\rm ;}
	\item the number of extremal Betti numbers of $K[V(G_{p,r,1})]/I(G_{p,r,1})$ is equal to $1 = b$.
\end{itemize}

}
\end{Example}

\medskip

\noindent
{\bf Acknowledgment.}
The authors are grateful to the referee for reading the manuscript carefully. 
The authors were partially supported by JSPS KAKENHI 26220701, 15K17507 and 17K14165.

\end{document}